\documentclass{amsart}
\usepackage{hyperref, amssymb}
\usepackage[all]{xy}

\theoremstyle{plain}
\newtheorem{theorem}{Theorem}[section]
\newtheorem{proposition}[theorem]{Proposition}
\newtheorem{lemma}[theorem]{Lemma}

\theoremstyle{definition}

\newtheorem{remark}[theorem]{Remark}

\newcommand{\f}{\varphi}

\newcommand{\CC}{\mathbb C}
\newcommand{\PP}{\mathbb P}

\newcommand{\BB}{\mathbf B}
\newcommand{\DD}{\mathbf D}
\newcommand{\EE}{\mathbf E}
\newcommand{\GG}{\mathbf G}

\newcommand{\MM}{\mathbf M}

\newcommand{\RR}{\mathbf R}
\newcommand{\WW}{\mathbf W}

\newcommand{\E}{{\mathcal E}}
\newcommand{\F}{{\mathcal F}}
\newcommand{\G}{{\mathcal G}}
\newcommand{\I}{{\mathcal I}}
\def\O{\mathcal O}

\newcommand{\Ker}{{\mathcal Ker}}
\newcommand{\Coker}{{\mathcal Coker}}

\newcommand{\Tor}{{\mathcal Tor}}

\newcommand{\Aut}{\operatorname{Aut}}
\newcommand{\Ext}{\operatorname{Ext}}
\newcommand{\ext}{\operatorname{ext}}
\newcommand{\Hom}{\operatorname{Hom}}
\newcommand{\h}{\operatorname{h}}
\def\H{\operatorname{H}}

\newcommand{\Hilb}{\operatorname{Hilb}}
\newcommand{\M}{\operatorname{M}}

\newcommand{\reg}{\operatorname{reg}}
\newcommand{\sing}{\operatorname{sing}}

\newcommand{\dual}{{\scriptscriptstyle \operatorname{D}}}

\newcommand{\st}{{\scriptstyle \operatorname{s}}}

\newcommand{\free}{{\scriptstyle \operatorname{free}}}
\newcommand{\el}{{\scriptstyle \operatorname{ell}}}

\newcommand{\irr}{{\scriptstyle \operatorname{irr}}}

\newcommand{\tensor}{\otimes}
\newcommand{\isom}{\simeq}
\newcommand{\lra}{\longrightarrow}
\def\egal{\ar@{=}}

\newcommand{\ba}{\begin{array}}
\newcommand{\ea}{\end{array}}

\begin{document}

\title[Moduli of space sheaves with Hilbert polynomial 4$m$+1]
{Moduli of space sheaves with Hilbert polynomial 4$m$+1}

\author{Mario Maican}
\address{Institute of Mathematics of the Romanian Academy, Calea Grivitei 21, Bucharest 010702, Romania}

\email{maican@imar.ro}


\begin{abstract}
We investigate the moduli space of sheaves supported on space curves of degree $4$ and having Euler characteristic $1$.
We give an elementary proof of the fact that this moduli space consists of three irreducible components.
\end{abstract}


\maketitle

\section{Introduction and preliminaries}
\label{introduction}

Let $\M_{\PP^n}(rm+\chi)$ be the moduli space of Gieseker semi-stable sheaves on the complex projective space $\PP^n$
having Hilbert polynomial $P(m) = rm+\chi$.
Le Potier \cite{lepotier} showed that $\M_{\PP^2}(rm+\chi)$ is irreducible and, if $r$ and $\chi$ are coprime, is smooth.
For low multiplicity the homology of $\M_{\PP^2}(rm+\chi)$ has been studied in \cite{choi_chung_moduli, choi_chung_geometry}, by the
wall-crossing method, and in \cite{choi_maican, maican_international, maican_sciences} by the Bia\l{y}nicki-Birula method.
When $n > 2$ the moduli space is no longer irreducible.
Thus, according to \cite{freiermuth_trautmann}, $\M_{\PP^3}(3m+1)$ has two irreducible components meeting transversally.
The focus of this paper is the moduli space $\MM = \M_{\PP^3}(4m+1)$ of stable sheaves on $\PP^3$ with Hilbert polynomial $4m+1$.
This has already been investigated in \cite{choi_chung_maican} using wall-crossing, by relating $\MM$ to $\Hilb_{\PP^3}(4m+1)$.
The main result of \cite{choi_chung_maican} states that $\MM$ consists of three irreducible components,
denoted $\overline{\RR}$, $\overline{\EE}$, $\mathbf{P}$, of dimension $16$, $17$, respectively, $20$.
The generic sheaves in $\overline{\RR}$ are structure sheaves of rational quartic curves.
The generic sheaves in $\overline{\EE}$ are of the form $\O_E(P)$, where $E$ is an elliptic quartic curve and $P$ is a point on $E$.
The third irreducible component parametrizes the planar sheaves.

The purpose of this paper is to reprove the decomposition of $\MM$ into irreducible components without using the wall-crossing method, see Theorem \ref{main_theorem}.
We achieve this as follows.
Using the decomposition of $\Hilb_{\PP^3}(4m+1)$ into irreducible components, found in \cite{chen_nollet},
we show that the subset of $\MM$ of sheaves generated by a global section is irreducible, see Proposition \ref{R_irreducible}.
This provides our first irreducible component. We then describe the sheaves having support an elliptic quartic curve,
see Section \ref{elliptic}.
To show that the set of such sheaves $\F$ is irreducible we use results from \cite{vainsencher} regarding the geometry of $\Hilb_{\PP^3}(4m)$.
Given $\F$, we construct at Proposition \ref{E_closure} a variety $\mathbf{W}$ together with a map $\sigma \colon \mathbf{W} \to \Gamma$,
the support map,
where $\Gamma \subset \Hilb_{\PP^3}(4m)$ is an irreducible quasi-projective curve, such that $\F \in \sigma^{-1}(x)$ for a point $x \in \Gamma$
and such that $\Gamma \setminus \{ x \}$ consists only of smooth curves. Moreover, the fibers of $\sigma$ are irreducible,
hence $\mathbf{W}$ is irreducible, and hence $\F$ is contained in the closure of the set of sheaves with support smooth elliptic curves.
Thus we obtain the second irreducible component.
The set $\mathbf{P}$ of planar sheaves is irreducible
because it is a bundle over the Grassmannian of planes in $\PP^3$ with fiber $\M_{\PP^2}(4m+1)$, which is, as mentioned above, irreducible.

We also rely on the cohomological classification of sheaves in $\MM$ found at \cite[Theorem 6.1]{choi_chung_maican},
which does not use the wall-crossing method (it uses the Beilinson spectral sequence).
We fix a $4$-dimensional vector space $V$ over $\CC$ and we identify $\PP^3$ with $\PP(V)$.
We fix a basis $\{ X, Y, Z, W \}$ of $V^*$.
We quote below \cite[Theorem 6.1]{choi_chung_maican}:

\begin{theorem}
\label{homological_conditions}
Let $\F$ give a point in $\M_{\PP^3}(4m+1)$.
Then $\F$ satisfies one of the following cohomological conditions:
\begin{enumerate}
\item[(i)]
$\h^0(\F \tensor \Omega^2(2)) = 0$, $\h^0(\F \tensor \Omega^1(1)) = 0$, $\h^0(\F) = 1$;
\item[(ii)]
$\h^0(\F \tensor \Omega^2(2)) = 0$, $\h^0(\F \tensor \Omega^1(1)) = 1$, $\h^0(\F) = 1$;
\item[(iii)]
$\h^0(\F \tensor \Omega^2(2)) = 1$, $\h^0(\F \tensor \Omega^1(1)) = 3$, $\h^0(\F) = 2$.
\end{enumerate}
\end{theorem}

Let $\MM_0$, $\MM_1$, $\MM_2 \subset \MM$ be the subsets of sheaves satisfying conditions (i), (ii), respectively, (iii).
We will call them \emph{strata}.
Clearly, $\MM_0$ is open, $\MM_1$ is locally closed and $\MM_2$ is closed.
We also quote the classification of the sheaves in each stratum in terms of locally free resolutions,
which was carried out at \cite[Theorem 6.1]{choi_chung_maican}.
The sheaves in $\MM_0$ are precisely the sheaves having a resolution of the form
\begin{equation}
\label{sheaves_in_R}
0 \lra 3\O(-3) \stackrel{\psi}{\lra} 5\O(-2) \stackrel{\f}{\lra} \O(-1) \oplus \O \lra \F \lra 0
\end{equation}
\[
\f = \left[
\ba{ccccc}
X & Y & Z & W & 0 \\
q_1 & q_2 & q_3 & q_4 & q_5
\ea
\right]
\]
or a resolution of the form
\begin{equation}
\label{sheaves_in_E}
0 \lra 3\O(-3) \stackrel{\psi}{\lra} 5\O(-2) \stackrel{\f}{\lra} \O(-1) \oplus \O \lra \F \lra 0
\end{equation}
\[
\f = \left[
\ba{ccccc}
l_1 & l_2 & l_3 & 0 & 0 \\
q_1 & q_2 & q_3 & q_4 & q_5
\ea
\right]
\]
where $l_1$, $l_2$, $l_3$ are linearly independent.
Let $\RR, \EE \subset \MM_0$ be the subsets of sheaves having resolution (\ref{sheaves_in_R}), respectively, (\ref{sheaves_in_E}).
Clearly, $\RR$ is an open subset of $\MM$ and consists of structure sheaves of rational quartic curves.
The set $\EE$ contains all extensions of $\CC_P$ by $\O_E$, where $E$ is an elliptic quartic curve and $P$ is a point on $E$.
The sheaves in $\MM_1$  are precisely the sheaves having a resolution of the form
\begin{equation}
\label{sheaves_in_M_1}
0 \lra 3\O(-3) \stackrel{\psi}{\lra} 5\O(-2) \oplus \O(-1) \stackrel{\f}{\lra} 2\O(-1) \oplus \O \lra \F \lra 0
\end{equation}
where $\f_{12} = 0$ and $\f_{11} \colon 5\O(-2) \to 2\O(-1)$ is not equivalent to a morphism represented by a matrix of the form
\[
\left[
\ba{ccccc}
\star & \star & 0 & 0 & 0 \\
\star & \star & \star & \star & \star
\ea
\right] \qquad \text{or} \qquad \left[
\ba{ccccc}
\star & \star & \star & \star & 0 \\
\star & \star & \star & \star & 0
\ea
\right].
\]
The sheaves in $\MM_2$ are precisely the sheaves of the form $\O_C(-P)(1)$, where $\O_C(-P)$ in $\O_C$
denotes the ideal sheaf of a closed point $P$ in a planar quartic curve $C$.

Assume now that $\F$ has resolution (\ref{sheaves_in_R}).
Let $S \subset \PP^3$ be the quadric surface given by the equation $q_5 = 0$.
From the snake lemma we get the resolution
\[
0 \lra 3\O(-3) \lra \Omega^1(-1) \lra \O_S \lra \F \lra 0.
\]
We consider first the case when $S$ is smooth.
The semi-stable sheaves on a smooth quadric surface with Hilbert polynomial $4m+1$ have been investigated in \cite{ballico_huh}.
We cite below the main result of \cite{ballico_huh}:

\begin{proposition}
\label{smooth_quadric}
Let $\F$ be a coherent sheaf on $\PP^1 \times \PP^1$ that is semi-stable relative to the polarization $\O(1, 1)$
and such that $P_{\F}(m) = 4m +1$.
Then precisely one of the following is true:
\begin{enumerate}
\item[(i)] $\F$ is the structure sheaf of a curve of type $(1,3)$;
\item[(ii)] $\F$ is the structure sheaf of a curve of type $(3, 1)$;
\item[(iii)] $\F$ is a non-split extension $0 \to \O_E \to \F \to \CC_P \to 0$
for a curve $E$ in $\PP^1 \times \PP^1$ of type $(2, 2)$ and a point $P \in E$.
Such an extension is unique up to isomorphism and satisfies the condition $\H^1(\F) = 0$.
\end{enumerate}
Thus, $\M_{\PP^1 \times \PP^1}(4m+1)$ has three connected components.
Two of these, $\PP(\H^0(\O(1, 3)))$ and $\PP(\H^0(\O(3, 1)))$, are isomorphic to $\PP^7$.
The third one is smooth, has dimension $9$, and is isomorphic to the universal elliptic curve
in $\PP(\H^0(\O(2, 2))) \times (\PP^1 \times \PP^1)$.
The sheaves at \emph{(iii)} are precisely the sheaves having a resolution of the form
\[
0 \lra \O(-2, -1) \oplus \O(-1, -2) \stackrel{\f}{\lra} \O(-1, -1) \oplus \O \lra \F \lra 0
\]
with $\f_{11} \neq 0$, $\f_{12} \neq 0$.
\end{proposition}

The following well-known lemma provides one of our main technical tools.

\begin{lemma}
Let $X$ be a projective scheme and $Y$ a subscheme.
Let $\F$ be a coherent $\O_X$-module and let $\G$ be a coherent $\O_Y$-module.
Then there is an exact sequence of vector spaces
\begin{multline}
\label{ext_sequence_1}
0 \lra \Ext^1_{O_Y}(\F_{|Y}, \G) \lra \Ext^1_{\O_X}(\F, \G) \lra \Hom_{\O_Y}(\Tor_1^{\O_X}(\F, \O_Y), \G) \\
\lra \Ext^2_{\O_Y}(\F_{|Y}, \G) \lra \Ext^2_{\O_X}(\F, \G).
\end{multline}
In particular, if $\F$ is an $\O_Y$-module, then the above exact sequence takes the form
\begin{multline}
\label{ext_sequence_2}
0 \lra \Ext^1_{O_Y}(\F, \G) \lra \Ext^1_{\O_X}(\F, \G) \lra \Hom_{\O_Y}(\F \tensor_{\O_X} \I_Y, \G) \\
\lra \Ext^2_{\O_Y}(\F, \G) \lra \Ext^2_{\O_X}(\F, \G).
\end{multline}
\end{lemma}


\section{Sheaves supported on rational quartic curves}
\label{rational}

Let $\RR_0 \subset \RR$ be the set of isomorphism classes of structure sheaves $\O_R$ of curves $R \subset S$
of type $(1, 3)$ or $(3, 1)$ on smooth quadrics $S \subset \PP^3$.
A curve of type $(1, 3)$ on $S$ can be deformed inside $\PP^3$ to a curve of type $(3, 1)$, hence $\RR_0$ is irreducible of dimension $16$.
Let $\EE_0 \subset \EE$ be the set of isomorphism classes of non-split extensions of $\CC_P$ by $\O_E$
for $E \subset S$ a curve of type $(2, 2)$ on a smooth quadric $S \subset \PP^3$ and $P$ a closed point on $E$.
From (\ref{ext_sequence_2}) and Proposition \ref{smooth_quadric} (iii) we have the exact sequence
\[
0 \lra \Ext^1_{\O_S} (\CC_P, \O_E) \isom \CC \lra \Ext^1_{\O_{\PP^3}} (\CC_P, \O_E) \lra \Hom_{\O_S}^{} (\CC_P, \O_E) = 0.
\]
We denote by $\O_E(P)$ the unique non-split extension of $\CC_P$ by $\O_E$.
Clearly, $\EE_0$ is irreducible of dimension $17$.
Let $\EE_\free \subset \EE_0$ denote the open subset of sheaves that are locally free on their schematic support,
which is equivalent to saying that $P \in \reg(E)$.
Let $\mathbf{P} \subset \M_{\PP^3}(4m+1)$ be the closed set of planar sheaves. It has dimension $20$.
Let $\mathbf{P}_\free \subset \mathbf{P}$ be the open subset of sheaves that are locally free on their support.
According to \cite{iena}, $\mathbf{P} \setminus \mathbf{P}_\free$ has codimension $2$ in $\mathbf{P}$.

\begin{proposition}
The closed sets $\overline{\RR}_0$, $\overline{\EE}_0$ and $\mathbf{P}$ are irreducible components of $\M_{\PP^3}(4m+1)$.
Moreover, $\RR_0$, $\EE_\free$ and $\mathbf{P}_\free$ are smooth open subsets of the moduli space.
\end{proposition} 

\begin{proof}
Let $\F = \O_R$ give a point in $\RR_0$, where $R \subset S$ is a curve of, say, type $(1, 3)$.
From Serre duality we have
\[
\Ext^2_{\O_S}(\F, \F) \isom \Hom_{\O_S}^{}(\F, \F(-2, -2))^* = 0.
\]
From the exact sequence (\ref{ext_sequence_2}) we get the relation
\[
\ext^1_{\O_{\PP^3}}(\F, \F) = \ext^1_{\O_S}(\F, \F) + \hom_{\O_S}^{}(\F(-2), \F) = 7 + \h^0(\O_R(2, 2)) = 16.
\]
This shows that $\overline{\RR}_0$ is an irreducible component of $\MM$ and that $\RR_0$ is smooth.

Consider next $\F = \O_E(P)$ giving a point in $\EE_0$.
As above, we have the relation
\[
\ext^1_{\O_{\PP^3}}(\F, \F) = \ext^1_{\O_S}(\F, \F) + \hom_{\O_S}^{}(\F(-2), \F) = 9 + \hom_{\O_S}^{}(\F, \F(2, 2)).
\]
Assume, in addition, that $\F$ is locally free on $E$.
Its rank must be $1$ because $E$ is a curve of multiplicity $4$.
Thus
\[
\Hom_{\O_S}^{}(\F, \F(2, 2)) \isom \H^0(\O_E(2, 2)) \isom \CC^8,
\]
hence $\ext^1_{\O_{\PP^3}}(\F, \F) = 17$.
This shows that $\overline{\EE}_0$ is an irreducible component of $\MM$ and that $\EE_\free$ is smooth.

Assume now that $\F$ is supported on a planar quartic curve $C \subset H$.
Using Serre duality and (\ref{ext_sequence_2}) we get the relation
\[
\ext^1_{\O_{\PP^3}}(\F, \F) = \ext^1_{\O_H}(\F, \F) + \hom_{\O_H}^{}(\F(-1), \F) = 17 + \hom_{\O_H}^{}(\F, \F(1)).
\]
Assume, in addition, that $\F$ is locally free on $C$, so a line bundle.
Thus
\[
\Hom_{\O_H}^{}(\F, \F(1)) \isom \H^0(\O_C(1)) \isom \CC^3,
\]
hence $\ext^1_{\O_{\PP^3}}(\F, \F) = 20$.
This shows that $\mathbf{P}$ is an irreducible component of $\MM$ and that $\mathbf{P}_\free$ is smooth.
\end{proof}

\begin{remark}
\label{planar_sheaves}
Let $\F$ be a one-dimensional sheaf on $\PP^3$ without zero-dimensional torsion.
Let $\F'$ be a planar subsheaf such that $\F/\F'$ has dimension zero. Then $\F$ is planar.
Indeed, say that $\F'$ is an $\O_H$-module for a plane $H \subset \PP^3$.
From (\ref{ext_sequence_1}) we have the exact sequence
\[
0 \to \Ext^1_{\O_H}((\F/\F')_{| H}, \F') \to \Ext^1_{\O_{\PP^3}}(\F/\F', \F') \to \Hom_{\O_H}(\Tor_1^{\O_{\PP^3}}(\F/\F', \O_H), \F').
\]
The group on the right vanishes because $\Tor_1^{\O_{\PP^3}}(\F/\F', \O_H)$ is supported on finitely many points,
yet $\F'$ has no zero-dimensional torsion. Thus $\F \in \Ext^1_{\O_H}((\F/\F')_{| H}, \F')$, so $\F$ is an $\O_H$-module.
\end{remark}

\begin{proposition}
\label{non-planar_M_1}
The non-planar sheaves in $\M_{\PP^3}(4m+1)$ having resolution (\ref{sheaves_in_M_1}) are precisely the non-split
extensions of the form
\begin{equation}
\label{sheaves_in_D'}
0 \lra \O_C \lra \F \lra \O_L \lra 0
\end{equation}
where $C$ is a planar cubic curve and $L$ is a line meeting $C$ with multiplicity $1$.
For such a sheaf, $\H^0(\F)$ generates $\O_C$.
The set $\RR$ consists precisely of the sheaves generated by a global section.
The set $\EE$ consists precisely of the sheaves $\F$ such that $\H^0(\F)$ generates a subsheaf with Hilbert polynomial $4m$.
\end{proposition}

\begin{proof}
Let $\f$ be a morphism as at (\ref{sheaves_in_M_1}).
Denote $\G = \Coker(\f_{11})$ and let $H \subset \PP^3$ be the plane given by the equation $\f_{22} = 0$.
From the snake lemma we have the exact sequence
\[
\O_H \lra \F \lra \G \lra 0.
\]
We examine first the case when
\[
\f_{11} \nsim \left[
\ba{ccccc}
0 & 0 & \star & \star & \star \\
\star & \star & \star & \star & \star
\ea
\right]. \quad \text{Thus we may write} \quad \f_{11} = \left[
\ba{ccccc}
X & Y & Z & W & 0 \\
0 & l_1 & l_2 & l_3 & l_4
\ea
\right].
\]
If $l_4$ is a multiple of $X$, then  $P_{\G} = 3$ (see the proof of \cite[Theorem 6.1(iii)]{choi_chung_maican}), hence,
by Remark \ref{planar_sheaves}, $\F$ is planar.
Assume now that $l_4$ is not a multiple of $X$ and let $L \subset \PP^3$ be the line given by the equations
$X = 0$, $l_4 = 0$.
Then $\G$ is a proper quotient sheaf of $\O_L(-1)$, hence it has support of dimension zero, and hence,
by Remark \ref{planar_sheaves}, $\F$ is planar.
It remains to examine the case when
\[
\f_{11} = \left[
\ba{ccccc}
u_1 & u_2 & u_3 & 0 & 0 \\
0 & v_1 & v_2 & v_3 & v_4
\ea
\right].
\]
Let $P$ be the point given by the ideal $(u_1, u_2, u_3)$ and let $L$ be the line given by the equations $v_3 = 0$, $v_4 = 0$.
We have an exact sequence
\[
\O_L(-1) \lra \G \lra \CC_P \lra 0.
\]
If the first morphism is not injective, then $\G$ has dimension zero, hence $\F$ is planar.
If $\G$ is an extension of $\CC_P$ by $\O_L(-1)$, then this extension does not split, otherwise $\O_L(-1)$ would be a destabilizing
quotient sheaf of $\F$. Thus, $\G \isom \O_L$ and we have an exact sequence
\[
0 \lra \E \lra \F \lra \O_L \lra 0
\]
where $\E$ gives a point in $\M_H(3m)$ and is generated by a global section.
Thus $\E$ is the structure sheaf of a cubic curve $C \subset H$.
If $L \subset H$, then from (\ref{ext_sequence_2}) we would have the exact sequence
\[
0 \lra \Ext^1_{\O_H} (\O_L, \O_C) \lra \Ext^1_{\O_{\PP^3}} (\O_L, \O_C) \lra \Hom_{\O_H}(\O_L(-1), \O_C).
\]
The group on the right vanishes because $\O_C$ is stable. We deduce that $\F$ lies in $\Ext^1_{\O_H} (\O_L, \O_C)$,
hence $\F$ is planar.

Thus far we have showed that if $\F$ is non-planar and has resolution (\ref{sheaves_in_M_1}), then $\F$ is an extension
as in the proposition. Conversely, given a non-split extension (\ref{sheaves_in_D'}), then $\F$ is semi-stable,
because $\O_C$ and $\O_L$ are stable.
In view of Theorem \ref{homological_conditions}, since $\F$ is non-planar, we have $\h^0(\F) = 1$.
Thus $\H^0(\F)$ generates $\O_C$.
It follows that $\F$ cannot have resolutions (\ref{sheaves_in_R}) or (\ref{sheaves_in_E}), otherwise $\H^0(\F)$
would generate $\F$ or would generate a subsheaf with Hilbert polynomial $4m$.
We conclude that $\F$ has resolution (\ref{sheaves_in_M_1}).

The rest of the proposition follows from Theorem \ref{homological_conditions}
and from the fact, proved in \cite{drezet_maican}, that for a planar sheaf $\F$ having resolution (\ref{sheaves_in_M_1}),
the space of global sections generates a subsheaf with Hilbert polynomial $4m-2$ or it generates the structure sheaf of a cubic curve.
\end{proof}

\begin{proposition}
\label{R_irreducible}
The set $\RR$ of sheaves in $\M_{\PP^3}(4m+1)$ generated by a global section is irreducible.
\end{proposition}

\begin{proof}
Let $\Hilb_{\PP^3}(4m+1)^\st \subset \Hilb_{\PP^3}(4m+1)$ be the open subset of semi-stable quotients.
The image of the canonical map
\[
\Hilb_{\PP^3}(4m+1)^\st \lra \M_{\PP^3}(4m+1)
\]
is $\RR$.
According to \cite[Theorem 4.9]{chen_nollet},
$\Hilb_{\PP^3}(4m+1)$ has four irreducible components, denoted $H_1$, $H_2$, $H_3$, $H_4$.
The generic point in $H_1$ is a rational quartic curve.
The generic curve in $H_2$ is the disjoint union of a planar cubic and a line.
The generic member of $H_3$ is the disjoint union of a point and an elliptic quartic curve.
The generic member of $H_4$ is the disjoint union of a planar quartic curve and three distinct points.
Thus, $H_2 \cup H_3 \cup H_4$ lies in the closed subset
\[
H = \{ [\O \twoheadrightarrow \mathcal{S}] \mid \ \h^0(\mathcal{S}) \ge 2 \} \subset \Hilb_{\PP^3}(4m+1).
\]
According to Theorem \ref{homological_conditions}, $H^\st = \emptyset$.
Indeed, any sheaf in $\MM_2$ cannot be generated by a single global section.
Thus, $\Hilb_{\PP^3}(4m+1)^\st$ is an open subset of $H_1$, hence it is irreducible, and hence $\RR$ is irreducible.
\end{proof}


\section{Sheaves supported on elliptic quartic curves}
\label{elliptic}

We will next examine the sheaves $\F$ having resolution (\ref{sheaves_in_E}).
Let $P$ be the point given by the ideal $(l_1, l_2, l_3)$.
Notice that the subsheaf of $\F$ generated by $\H^0(\F)$ is the kernel of the canonical map $\F \to \CC_P$.
This shows that $\F$ is non-planar because, according to \cite{drezet_maican},
the global sections of a sheaf in $\M_{\PP^2}(4m+1)$ whose first cohomology vanishes
generate a subsheaf with Hilbert polynomial $4m - 2$ or the structure sheaf of a planar cubic curve, which is not the case here.
We consider first the case when $q_4$ and $q_5$ have no common factor, so they define a curve $E$.
Applying the snake lemma to the diagram
\[
\xymatrix
{
& & 0 \ar[d] & 0 \ar[d] \\
0 \ar[r] & \O(-4) \ar[r]^-{\tiny \left[ \!\!\! \ba{c} q_5 \\ q_4 \ea \!\!\! \right]} & 2\O(-2) \ar[r]^-{[-q_4 \ q_5]} \ar[d] & \O \ar[r] \ar[d] & \O_E \ar[r] & 0 \\
0 \ar[r] & 3\O(-3) \ar[r] & 5\O(-2) \ar[r]^-{\f} \ar[d] & \O(-1) \oplus \O \ar[r] \ar[d] & \F \ar[r] & 0 \\
0 \ar[r] & \mathcal{K} \ar[r] & 3\O(-2) \ar[r]^-{[l_1 \ l_2 \ l_3]} \ar[d] & \O(-1) \ar[r] \ar[d] & \CC_P \ar[r] & 0 \\
& & 0 & 0
}
\]
we see that $\F$ is an extension of $\CC_P$ by $\O_E$. 
From Serre duality we have
\[
\Ext^1_{\O_{\PP^3}}(\CC_P, \O_E) \isom \Ext^2_{\O_{\PP^3}}(\O_E, \CC_P)^* \isom \CC.
\]
The group in the middle can be determined by applying $\Hom(\rule{7pt}{.5pt}, \CC_P)$ to the first row of the diagram above.
We may write $\F = \O_E(P)$.

\begin{proposition}
\label{F_stable}
The sheaf $\O_E(P)$ is stable.
\end{proposition}

\begin{proof}
We will show that $\O_E$ is stable, forcing $\O_E(P)$ to be stable.
To prove that $\O_E$ is stable, we must show that it does not contain a stable subsheaf $\E$ having one of the following Hilbert polynomials:
$m$, $m+1$ (i.e. the structure sheaf of a line), $2m$, $2m+1$ (i.e. the structure sheaf of a conic curve), $3m$, $3m+1$.
The structure sheaf of a line contains subsheaves having Hilbert polynomial $m$ and the structure sheaf of a conic curve contains subsheaves
having Hilbert polynomial $2m$. Thus, it is enough to consider only the Hilbert polynomials $m$, $2m$, $3m+1$, $3m$.
In the first case, we have a commutative diagram
\[
\xymatrix
{
0 \ar[r] & \O(-3) \ar[d]^-{\gamma} \ar[r] & 2\O(-2) \ar[r] \ar[d]^-{\beta} & \O(-1) \ar[r] \ar[d]^-{\alpha} & \E \ar[r] \ar[d] & 0 \\
0 \ar[r] & \O(-4) \ar[r] & 2\O(-2) \ar[r] & \O \ar[r] & \O_E \ar[r] & 0
}
\]
in which $\alpha \neq 0$. It follows that $\O(-3) \isom \Ker(\gamma) \isom \Ker(\beta)$, which is absurd.
In the second case, we get a commutative diagram
\[
\xymatrix
{
0 \ar[r] & 2\O(-3) \ar[d]^-{\gamma} \ar[r] & 4\O(-2) \ar[r] \ar[d]^-{\beta} & 2\O(-1) \ar[r] \ar[d]^-{\alpha} & \E \ar[r] \ar[d] & 0 \\
0 \ar[r] & \O(-4) \ar[r] & 2\O(-2) \ar[r] & \O \ar[r] & \O_E \ar[r] & 0
}
\]
in which $\alpha \neq 0$, hence $\Ker(\alpha) \isom \O(-1)$ or $\O(-2)$.
From the exact sequence
\[
0 \lra 2\O(-3) \isom \Ker(\gamma) \lra \Ker(\beta) \lra \Ker(\alpha) \lra \Coker(\gamma) \isom \O(-4)
\]
we see that $\Ker(\beta) \isom 3\O(-2)$ and we get the exact sequence
\[
0 \lra 2\O(-3) \lra 3\O(-2) \lra \Ker(\alpha) \lra 0.
\]
Such an exact sequence cannot exist.
In the third case, we use the resolution of $\E$ given at \cite[Theorem 1.1]{freiermuth_trautmann}.
We obtain a commutative diagram
\[
\xymatrix
{
0 \ar[r] & 2\O(-3) \ar[d]^-{\gamma} \ar[r] & 3\O(-2) \oplus \O(-1) \ar[r] \ar[d]^-{\beta} & \O(-1) \oplus \O \ar[r] \ar[d]^-{\alpha} & \E \ar[r] \ar[d] & 0 \\
0 \ar[r] & \O(-4) \ar[r] & 2\O(-2) \ar[r] & \O \ar[r] & \O_E \ar[r] & 0
}
\]
in which $\alpha$ is non-zero on global sections, hence $\Ker(\alpha) \isom \O(-1)$.
We obtain a contradiction from the exact sequence
\[
0 \lra 2\O(-3) \isom \Ker(\gamma) \lra \Ker(\beta_{11}) \oplus \O(-1) \lra \Ker(\alpha) \lra 0.
\]
Assume, finally, that $\E$ gives a stable point in $\M_{\PP^3}(3m)$. If $\H^0(\E) \neq 0$, then it is easy to see that $\E$ is the structure sheaf
of a planar cubic curve, hence we get a commutative diagram
\[
\xymatrix
{
0 \ar[r] & \O(-4) \ar[d]^-{\gamma} \ar[r] & \O(-3) \oplus \O(-1) \ar[r] \ar[d]^-{\beta} & \O \ar[r] \ar[d]^-{\alpha} & \E \ar[r] \ar[d] & 0 \\
0 \ar[r] & \O(-4) \ar[r] & 2\O(-2) \ar[r] & \O \ar[r] & \O_E \ar[r] & 0
}
\]
in which $\alpha$ is injective. We get a contradiction from the fact that $\O(-1)$ is a subsheaf of $\Ker(\beta) \isom \Ker(\gamma)$.
If $\H^0(\E) = 0$, then we get a commutative diagram of the form
\[
\xymatrix
{
0 \ar[r] & 3\O(-3) \ar[d]^-{\gamma} \ar[r] & 6\O(-2) \ar[r] \ar[d]^-{\beta} & 3\O(-1) \ar[r] \ar[d]^-{\alpha} & \E \ar[r] \ar[d] & 0 \\
0 \ar[r] & \O(-4) \ar[r] & 2\O(-2) \ar[r] & \O \ar[r] & \O_E \ar[r] & 0
}
\]
It is easy to see that $\alpha(1)$ is injective on global sections, hence $\Coker(\alpha)$ is isomorphic to the structure sheaf of a point
and $\Coker(\beta) \isom \O(-2)$. We get a contradiction from the exact sequence
\[
\O(-4) \isom \Coker(\gamma) \lra \Coker(\beta) \lra \Coker(\alpha). \qedhere
\]
\end{proof}

To finish the discussion about sheaves at Theorem \ref{homological_conditions} (i), we need to examine the case when
$q_4 = u v_1$ and $q_5 = u v_2$ with linearly independent $v_1, v_2 \in V^*$.
Let $H$ be the plane given by the equation $u = 0$ and $L$ the line given by the equations $v_1 = 0$, $v_2 = 0$.
We apply the snake lemma to the diagram
\[
\xymatrix
{
 & & 0 \ar[d] & & 0 \ar[d] \\
 0 \ar[r] & \O(-3) \ar[r] & 2\O(-2) \ar[rr]^-{[v_1 \ v_2]} \ar[d] & &
 \O(-1) \ar[r] \ar[d]^-{\tiny \left[ \!\! \ba{c} 0 \\ u \ea \!\! \right]} & \O_L(-1) \ar[r] & 0 \\
 0 \ar[r] & 3\O(-3) \ar[r] & 5\O(-2) \ar[rr]^-{\f} \ar[d] & & \O(-1) \oplus \O \ar[r] \ar[d] & \F \ar[r] & 0 \\
 0 \ar[r] & \mathcal{K} \ar[r] & 3\O(-2) \ar[rr]^-{\tiny \left[ \!\! \ba{ccc} l_1 & l_2 & l_3 \\ \star & \star & \star \ea \!\! \right]} \ar[d] & &
 \O(-1) \oplus \O_H \ar[r] \ar[d] & \G \ar[r] & 0 \\
  & & 0 & & 0
}
\]
The kernel of the canonical map $\G \to \CC_P$ is an $\O_H$-module.
This shows that $\F$ is not isomorphic to $\G$, otherwise, in view of Remark \ref{planar_sheaves}, $\F$ would be planar.
Thus $\O_L(-1) \to \F$ is non-zero, hence it is injective.
We get a non-split extension
\begin{equation}
\label{sheaves_in_D}
0 \lra \O_L(-1) \lra \F \lra \G \lra 0
\end{equation}
and it becomes clear that $P \in H$ and that $\G$ gives a point in $\M_{\PP^3} (3m+1)$.
From Remark \ref{planar_sheaves} we see that $\G$ gives a point in $\M_H (3m+1)$.
Thus, $\G$ is the unique non-split extension of $\CC_P$ by $\O_C$ for a cubic curve $C \subset H$ containing $P$.
We write $\G = \O_C(P)$.
Let $\DD \subset \M_{\PP^3}(4m+1)$ be the set of non-split extension sheaves as in (\ref{sheaves_in_D})
that are non-planar (we allow the possibility that $L \subset H$, in which case the support of $\F$ is contained in the double plane $2H$).

We examine first the case when $L \nsubseteq H$, that is, $L$ meets $C$ with multiplicity $1$, at a point $P'$.
According to \cite[Theorem 1.1]{freiermuth_trautmann} there is a resolution
\begin{equation}
\label{planar_3m+1}
0 \lra 2\O(-3) \stackrel{\delta}{\lra} 3\O(-2) \oplus \O(-1) \stackrel{\gamma}{\lra} \O(-1) \oplus \O \lra \G \lra 0
\end{equation}
\[
\delta = \left[
\ba{ll}
\phantom{-} u & \phantom{-} 0 \\
\phantom{-} 0 & \phantom{-} u \\
-u_1 & -u_2 \\
-g_1 & -g_2
\ea
\right], \qquad \gamma = \left[
\ba{cccc}
u_1 & u_2 & u & 0 \\
g_1 & g_2 & 0 & u
\ea
\right]
\]
where $\operatorname{span} \{ u_1, u_2, u \} = \operatorname{span} \{ l_1, l_2, l_3 \}$
and $C$ has equation $u_1 g_2 - u_2 g_1 = 0$ in $H$.
Note that $\G_{| L} \isom \CC_{P'}$ unless $\gamma(P') = 0$, in which case $\G_{| L} \isom \CC_{P'} \oplus \CC_{P'}$.
But $\gamma(P') = 0$ if and only if $P' = P \in \sing(C)$.
From (\ref{ext_sequence_1}) we have the exact sequence
\[
0 \to \Ext^1_{\O_L} (\G_{| L}, \O_L(-1)) \to \Ext^1_{\O_{\PP^3}} (\G, \O_L(-1)) \to
\Hom_{\O_L} (\Tor_1^{\O_{\PP^3}}(\G, \O_L), \O_L(-1)).
\]
The group on the right vanishes because $\O_L(-1)$ has no zero-dimensional torsion.
It follows that
\[
\Ext^1_{\O_{\PP^3}} (\G, \O_L(-1)) \isom
\begin{cases}
\CC & \text{if $P \neq P'$ or if $P = P' \in \reg(C)$}, \\
\CC^2 & \text{if $P = P' \in \sing(C)$}.
\end{cases}
\]
Let $\DD_0 \subset \DD$ be the open subset given by the conditions that $L \nsubseteq H$
and either $P \neq P'$ or $P = P' \in \reg(C)$. The map
\[
\DD_0 \lra \Hilb_{\PP^3}(m+1) \times \M_{\PP^3}(3m+1), \qquad [\F] \longmapsto (L, [\G])
\]
is injective and has irreducible image.
We deduce that $\DD_0$ is irreducible and has dimension $16$.

Let $\DD' \subset \M_{\PP^3}(4m+1)$ be the subset of non-split extensions (\ref{sheaves_in_D'}).
Denote $P = L \cap C$.
From (\ref{ext_sequence_1}) we have the exact sequence
\[
0 \to \CC \isom \Ext^1_{\O_H}(\CC_{P}, \O_C) \to \Ext^1_{\O_{\PP^3}}(\O_L, \O_C) \to
\Hom_{\O_H}(\Tor_1^{\O_{\PP^3}}(\O_L, \O_H), \O_C) = 0.
\]
We deduce that, given $L$ and $C$, there is a unique non-split extension of $\O_L$ by $\O_C$.
The map
\[
\DD' \lra \Hilb_{\PP^3}(m+1) \times \Hilb_{\PP^3}(3m)
\]
sending $\F$ to $(L, C)$ is injective and has irreducible image.
We deduce that $\DD'$ is irreducible and has dimension $15$.
Tensoring (\ref{sheaves_in_D'}) with $\O_H$ we get the exact sequence
\[
0 = \Tor_1^{\O_{\PP^3}} (\O_L, \O_H) \lra \O_C \lra \F_{| H} \lra \CC_{P} \lra 0
\]
from which we see that $\F_{| H} \isom \O_C(P)$. We obtain the extension
\[
0 \lra \O_L(-1) \lra \F \lra \O_C(P) \lra 0.
\]
We deduce that $[\F] \in \DD$.
Thus, $\DD' \subset \DD$. Moreover, $\DD' \cap \DD_0$ is open and non-empty in $\DD'$ because it consists precisely
of extensions as above for which $P \in \reg(C)$.
Thus, $\DD' \subset \overline{\DD}_0$.

\begin{remark}
Note that $\DD_0 \setminus \DD'$ is the open subset of $\DD$ given by the conditions $L \nsubseteq H$ and $P \neq P'$.
We claim that $\DD_0 \setminus \DD'$ is the set of sheaves of the form $\O_D(P)$, where $D = L \cup C$ is the union
of a line and a planar cubic curve having intersection of multiplicity $1$ and $P \in C \setminus L$.
First we show that the notation $\O_D(P)$ is justified.
From (\ref{ext_sequence_1}) we have the exact sequence
\begin{multline*}
0 \lra \CC \isom \Ext^1_{\O_L}(\CC_{P'}, \O_L(-1)) \lra \Ext^1_{\O_{\PP^3}}(\O_C, \O_L(-1)) \\
\lra \Hom(\Tor_1^{\O_{\PP^3}}(\O_C, \O_L), \O_L(-1)) = 0
\end{multline*}
which shows that $\O_D$ is the unique non-split extension of $\O_C$ by $\O_L(-1)$.
The long exact sequence of groups
\begin{multline*}
0 = \Ext^1_{\O_{\PP^3}}(\CC_P, \O_L(-1)) \lra \Ext^1_{\O_{\PP^3}}(\CC_P, \O_D) \lra
\Ext^1_{\O_{\PP^3}}(\CC_P, \O_C) \isom \CC \\
\lra \Ext^2_{\O_{\PP^3}} (\CC_P, \O_L(-1)) = 0
\end{multline*}
shows that there is a unique non-split extension of $\CC_P$ by $\O_D$, which we denote by $\O_D(P)$.
Given $\F \in \DD_0 \setminus \DD'$, the pull-back of $\O_C$ in $\F$, denoted $\F'$, is a non-split extension of $\O_C$ by $\O_L(-1)$.
Indeed, if $\F'$ were a split extension, then $\O_C \subset \F$ and $\F/\O_C \isom \O_L(-1) \oplus \CC_P$,
so $\O_L(-1)$ would be a destabilising quotient sheaf of $\F$. Thus $\F' \isom \O_D$ and $\F \isom \O_D(P)$.
Conversely, $\O_D(P) / O_L(-1)$ is an extension of $\CC_P$ by $\O_C$, hence $\O_D(P) / O_L(-1) \isom \O_C(P)$.
\end{remark}

\begin{remark}
\label{no_extensions}
If $L \cap C = \{ P \}$ is a regular point of $C$, and $D = L \cup C$, then
there are no semi-stable extensions of the form
\[
0 \lra \O_D \lra \F \lra \CC_P \lra 0.
\]
Indeed, if $\F$ were such a semi-stable extension, then we would also have an extension
\[
0 \lra \O_L(-1) \lra \F \lra \G \lra 0
\]
where $\G$ is an extension of $\CC_P$ by $\O_C$. Note that $\G$ is a non-split extension, otherwise $\O_C$
would be a destabilizing quotient sheaf of $\F$.
Thus $\F$ is the unique non-split extension of $\O_C(P)$ by $\O_L(-1)$,
so it is also the unique non-split extension of $\O_L$ by $\O_C$.
Thus $\H^0(\F)$ generates $\O_C$, hence $\O_D$ is a subsheaf of $\O_C$, which is absurd.
\end{remark}

\begin{remark}
\label{S_irreducible}
The set $\mathbf{S} \subset \M_{\PP^2}(3m) \times \M_{\PP^2}(3m+1)$ of pairs $([\E], [\G])$
such that $\H^0(\E) = 0$ and $\E$ is a subsheaf of $\G$ is irreducible.
By duality, this is equivalent to saying that the set $\mathbf{S}^\dual \subset \M_{\PP^2}(3m-1) \times \M_{\PP^2}(3m)$
of pairs $([\G], [\E])$ such that $\H^0(\E) = 0$ and $\G$ is a subsheaf of $\E$ is irreducible.
Given an exact sequence
\[
0 \lra \G \lra \E \lra \CC_{P'} \lra 0
\]
we may combine the resolutions of sheaves on $\PP^2$
\[
0 \lra \O(-3) \oplus \O(-2) \xrightarrow{\tiny \left[ \!\! \ba{cc} q_1 & \!\!\! u_1 \\ q_2 & \!\!\! u_2 \ea \!\! \right]} 2\O(-1) \lra \G \lra 0
\]
and
\[
0 \lra \O(-3) \lra 2\O(-2) \xrightarrow{\tiny \left[ \!\! \ba{cc} v_1 & \!\!\! v_2 \ea \!\! \right]} \O(-1) \lra \CC_{P'} \lra 0
\]
to form the resolution
\[
0 \lra \O(-3) \stackrel{\psi}{\lra} \O(-3) \oplus 3\O(-2) \stackrel{\f}{\lra} 3\O(-1) \lra \E \lra 0,
\]
\[
\f = \left[
\ba{cccc}
q_1 & u_1 & l_{11} & l_{12} \\
q_2 & u_2 & l_{21} & l_{22} \\
0 & 0 & v_1 & v_2
\ea
\right].
\]
We indicate by the index $i$ the maximal minor of a matrix obtained by deleting column $i$.
The condition $\H^0(\E) = 0$ is equivalent to the condition $\psi_{11} \neq 0$, which is equivalent to the following conditions:
$\f_1 \neq 0$ and $\f_1$ divides $\f_2$, $\f_3$, $\f_4$.
As $\f_1$ divides both $(q_1 u_2 - u_1 q_2) v_1$ and $(q_1 u_2 - u_1 q_2) v_2$, we see that $\f_1$ is a multiple of $q_1 u_2 - u_1 q_2$.
It follows that $\f$ is equivalent to the matrix
\[
\upsilon = \left[
\ba{cccc}
l_{11} v_2 - l_{12} v_1 & u_1 & l_{11} & l_{12} \\
l_{21} v_2 - l_{22} v_1 & u_2 & l_{21} & l_{22} \\
0 & 0 & v_1 & v_2
\ea
\right].
\]
Let $U \subset \Hom(\O(-3) \oplus 3\O(-2), 3\O(-1))$ be the set of morphisms represented by matrices
$\upsilon$ as above satisfying the following conditions:
$\upsilon_1 \neq 0$, $u_1$ and $u_2$ are linearly independent, $v_1$ and $v_2$ are linearly independent.
Clearly, $U$ is irreducible.
Let $\upsilon' \in \Hom(\O(-3) \oplus \O(-2), 2\O(-1))$ be the morphism represented by the matrix
\[
\left[
\ba{cc}
l_{11} v_2 - l_{12} v_1 & u_1 \\
l_{21} v_2 - l_{22} v_1 & u_2
\ea
\right].
\]
The above discussion shows that the map $\pi \colon U \to \mathbf{S}^\dual$, $\upsilon \mapsto ([\Coker(\upsilon')], [\Coker(\upsilon)])$
is surjective. Thus, $\mathbf{S}^\dual$ is irreducible.
The open subset $\mathbf{S}_\irr \subset \mathbf{S}$, given by the condition that the schematic support of $\G$ be irreducible,
is irreducible. 
\end{remark}

Let $\DD_1 \subset \DD$ be the locally closed subset given by the conditions $L \nsubseteq H$ and $P = P' \in \sing(C)$.
Since $\dim \Ext^1_{\O_{\PP^3}}(\G, \O_L(-1)) = 2$, we see that $\dim \DD_1 = 14$.
The set of cubic curves in $\PP^2$ that are singular at a fixed point is irreducible. It follows that $\DD_1$ is irreducible, as well.

\begin{proposition}
\label{D_1_in_D_0}
The set $\DD_1$ is contained in the closure of $\DD_0$.
\end{proposition}

\begin{proof}
Consider $[\F] \in \DD_0 \cup \DD_1$. Consider extension (\ref{sheaves_in_D}) in which $\G = \O_C(P)$ and $L \cap H = \{ P'\}$.
Dualizing we get the extension
\[
0 \lra \O_C(-P) \lra \F^\dual \lra \O_L(-1) \lra 0.
\]
Tensoring with $\O_H$ we get the exact sequence
\[
0 = \Tor_1^{\O_{\PP^3}} (\O_L(-1), \O_H) \lra \O_C(-P) \lra (\F^\dual)_{| H} \lra \CC_{P'} \lra 0.
\]
This short exact sequence does not split.
Indeed, by \cite{maican_duality}, $\F^\dual$ is stable and has slope $-1/4$, hence $\O_C(-P)$, which has slope $-1/3$,
cannot be a quotient sheaf of $\F^\dual$.
Since $\O_C(-P)$ is stable, it is easy to see that $(\F^\dual)_{| H}$ gives a sheaf in $\M_H(3m)$ supported on $C$.
The kernel of the map $\F^\dual \to (\F^\dual)_{| H}$ is supported on $L$ and has no zero-dimensional torsion,
hence it is isomorphic to $\O_L(-2)$.
Denote $\E = ((\F^\dual)_{| H})^\dual$. Dualizing the exact sequence
\[
0 \lra \O_L(-2) \lra \F^\dual \lra (\F^\dual)_{| H} \lra 0
\]
we obtain the extension
\begin{equation}
\label{E-F-O_L}
0 \lra \E \lra \F \lra \O_L \lra 0.
\end{equation}
Tensoring with $\O_H$, and taking into account the fact that $\Tor_1^{\O_{\PP^3}}(\O_L, \O_H) = 0$, we get the exact sequence
\begin{equation}
\label{E-O_C(P)}
0 \lra \E \lra \O_C(P) \lra \CC_{P'} \lra 0.
\end{equation}
From (\ref{ext_sequence_1}) we have the exact sequence
\[
0 \lra \Ext^1_{\O_H} (\CC_{P'}, \E) \stackrel{\epsilon}{\lra} \Ext^1_{\O_{\PP^3}}(\O_L, \E) \lra \Hom (\Tor_1^{\O_{\PP^3}}(\O_L, \O_H), \E) = 0.
\]
It is clear now that the isomorphism class of $\F$ corresponds to the isomorphism class of $\O_C(P)$ under the bijective map $\epsilon$.
Let $\DD'' \subset (\DD_0 \cup \DD_1) \setminus \DD'$ be the subset given by the condition that $C$ be irreducible.
Note that $\DD''$ is an open subset of $\DD$ and contains an open subset of $\DD_1$.
We will prove below that $\DD''$ is irreducible.
Since $\DD_1$ is irreducible, we arrive at the conclusion of the proposition:
\[
\DD_1 \subset \overline{\DD'' \cap \DD}_1 \subset \overline{\DD}{}'' = \overline{\DD'' \cap \DD}_0 \subset \overline{\DD}_0.
\]
Consider the subset
\[
\mathbf{S}'' \subset \Hilb_{\PP^3}(m+1) \times \M_{\PP^3}(3m) \times \M_{\PP^3}(3m+1)
\]
of triples $(L, [\E], [\G])$ satisfying the following conditions: $\E$ and $\G$ are supported on a planar irreducible cubic curve $C$,
$\H^0(\E) = 0$, $\E$ is a subsheaf of $\G$, and $L \cap C = \{ P' \}$, where $\CC_{P'} \isom \G/\E$.
Note that the projection $\mathbf{S}'' \to \M_{\PP^3}(3m) \times \M_{\PP^3}(3m+1)$ has fibers affine planes and has image
the irreducible variety  $\mathbf{S}_\irr$ from Remark \ref{S_irreducible}.
It follows that $\mathbf{S}''$ is irreducible.
To prove that $\DD''$ is irreducible, we will show that the morphism
\[
\eta \colon \DD'' \lra \mathbf{S}'', \qquad
\eta ([\F]) = (L, [((\F^\dual)_{| H})^\dual], [\F_{| H}])
\]
is bijective. We first verify surjectivity. Given an extension
\[
0 \lra \E \lra \G \lra \CC_{P'} \lra 0
\]
we let $\F \in \Ext^1_{\O_{\PP^3}} (\O_L, \E)$ be the image of $\G$ under $\epsilon$.
Since $\G$ does not split, neither does $\F$.
By hypothesis $\E$ has irreducible support, hence $\E$ is stable, and, a fortiori, $\F$ is stable.
Applying the snake lemma to the diagram
\[
\xymatrix
{
0 \ar[r] & \E \egal[d] \ar[r] & \F \ar[r] \ar[d] & \O_L \ar[d] \ar[r] & 0 \\
0 \ar[r] & \E \ar[r] & \G \ar[r] & \CC_{P'} \ar[r] & 0
}
\]
we get the extension
\[
0 \lra \O_L(-1) \lra \F \lra \G \lra 0.
\]
Thus, $[\F] \in \DD_0 \cup \DD_1$ and $\F_{| H} \isom \G$, where $H$ is the plane containing $C$.
Dualizing the first row of the above diagram we see that $(\F^\dual)_{| H} \isom \E^\dual$.
By hypothesis $\E$ is not isomorphic to $\O_C$, hence $[\F] \notin \DD'$.
Thus $[\F] \in \DD''$ and $\eta ([\F]) = (L, [\E], [\G])$.
This proves that $\eta$ is surjective.
Since $[\F] = \epsilon ([\G])$ we see that $\eta$ is also injective.
\end{proof}

\noindent
We will next examine the sheaves in $\DD$ for which $L \subset H$.
From (\ref{ext_sequence_2}) we have the exact sequence
\begin{multline*}
0 \lra \Ext^1_{\O_H} (\O_C(P), \O_L(-1)) \lra \Ext^1_{\O_{\PP^3}} (\O_C(P), \O_L(-1)) \\
\lra \Hom(\O_C(P)(-1), \O_L(-1)) \\
\lra \Ext^2_{\O_H} (\O_C(P), \O_L(-1)) \isom \Hom_{\O_H}^{}(\O_L(-1), \O_C(P)(-3))^* = 0.
\end{multline*}
Thus, we have non-planar sheaves precisely if $\Hom (\O_C(P), \O_L) \neq 0$.
Any non-zero morphism $\alpha \colon \O_C(P) \to \O_L$ fits into a commutative diagram
\[
\xymatrix
{
0 \ar[r] & 2\O_H(-2) \ar[r]^-{\upsilon} \ar[d]^-{\gamma} & \O_H(-1) \oplus \O_H \ar[d]^-{\beta} \ar[r] & \O_C(P) \ar[r] \ar[d]^-{\alpha} & 0 \\
0 \ar[r] & \O_H(-1) \ar[r]^-{l} & \O_H \ar[r] & \O_L \ar[r] & 0
}
\]
\[
\beta = \left[
\ba{cc}
v & c
\ea
\right], \qquad \gamma = \left[
\ba{cc}
v_1 & v_2
\ea
\right], \qquad \upsilon = \left[
\ba{cc}
u_1 & u_2 \\
g_1 & g_2
\ea
\right]
\]
with $\beta \neq 0$.
Note that $c \neq 0$, otherwise $\Coker(\beta)$ would be the structure sheaf of a line
and we would have the relation $(vu_1, vu_2) = (lv_1, lv_2)$.
Thus $v_1$ and $v_2$ would be linearly independent, hence $\Coker(\gamma)$ would be zero-dimensional,
and hence $\Coker(\beta)$ would be zero-dimensional, which is absurd.
Replacing, possibly, $\upsilon$ with an equivalent matrix, we may assume that $g_1$ and $g_2$ are divisible by $l$.
Conversely, if $\O_C(P)$ is the cokernel of the morphism
\[
\upsilon = \left[
\ba{cc}
u_1 & u_2 \\
l v_1 & l v_2
\ea
\right], \qquad \text{then, denoting} \qquad \upsilon' = \left[
\ba{cc}
u_1 & u_2 \\
v_1 & v_2
\ea
\right],
\]
we can apply the snake lemma to the commutative diagram
\begin{equation}
\label{upsilon_diagram}
\xymatrix
{
 & 2\O_H(-2) \egal[r] \ar[d]^-{\upsilon'} & 2\O_H(-2) \ar[d]^-{\upsilon} \\
 0 \ar[r] & 2\O_H(-1) \ar[r]^-{1 \oplus l} & \O_H(-1) \oplus \O_H \ar[r] & \O_L \ar[r] & 0
}
\end{equation}
to get a surjective map $\O_C(P) \to \O_L$. This discussion shows that $\Hom(\O_C(P), \O_L)$ does not vanish
precisely if $C = L \cup C'$ for a conic curve $C' \subset H$ and for $P \in C'$.
In this case we have a commutative diagram
\[
\xymatrix
{
\Hom(\O_C, \O_L(-1)) = 0 \ar[d]
\\
\Ext^1_{\O_H} (\CC_P, \O_L(-1)) \ar[d] & & \Hom(\CC_P, \O_L) = 0 \ar[d]
\\
\Ext^1_{\O_H} (\O_C(P), \O_L(-1)) \ar@{^(->}[r]  \ar[d] &
\Ext^1_{\O_{\PP^3}} (\O_C(P), \O_L(-1)) \ar@{->>}[r] \ar[d]^-{\delta} & \Hom(\O_C(P), \O_L) \ar[d]^-{\isom}
\\
\Ext^1_{\O_H} (\O_C, \O_L(-1)) \ar@{^(->}[r] \ar[d] & \Ext^1_{\O_{\PP^3}} (\O_C, \O_L(-1)) \ar@{->>}[r] & \Hom(\O_C, \O_L) \isom \CC 
\\
\Ext^2_{\O_H} (\CC_P, \O_L(-1)) \ar[r]^-{\isom} \ar[d] & \Hom_{\O_H}^{} (\O_L, \CC_P)^*
\\
\Ext^2_{\O_H} (\O_C(P), \O_L(-1)) = 0 
}
\] 
Here $\delta(\F)$ is the pull-back of $\O_C$ in $\F$.
If $P \notin L$, then $\delta$ is an isomorphism.
If $P \in L$, then we have an exact sequence
\[
0 \lra \CC \lra \Ext^1_{\O_{\PP^3}} (\O_C(P), \O_L(-1)) \stackrel{\delta}{\lra} \Ext^1_{\O_{\PP^3}} (\O_C, \O_L(-1)) \lra \CC \lra 0.
\]
If $\F$ is non-planar, then $\delta(\F)$ is generated by a global section.
Indeed, in view of Proposition \ref{non-planar_M_1}, $\F$ cannot have resolution (\ref{sheaves_in_M_1}),
so it has resolution (\ref{sheaves_in_R}) or (\ref{sheaves_in_E}).
Also, $\F$ is not generated by a global section because $\O_C(P)$ is not generated by a global section.
It follows that $P_{\F'}(m) = 4m$, where $\F' \subset \F$ is the subsheaf generated by $\H^0(\F)$.
But $\F'$ maps to $\O_C$, hence $\delta(\F) \subset \F'$.
These two sheaves have the same Hilbert polynomial, so they coincide.
We conclude that $\delta(\F)$ is the structure sheaf $\O_D$ of a quartic curve $D$.
If $P \notin L$, then $\F \isom \O_D(P)$.

Assume now that $P \in L$. The preimage of $[\O_D]$ under the induced map
\[
\PP \big( \Ext^1_{\O_{\PP^3}}(\O_C(P), \O_L(-1)) \big) \setminus \PP(\CC) \lra \PP \big( \Ext^1_{\O_{\PP^3}}(\O_C, \O_L(-1)) \big)
\]
is an affine line that maps to a curve in $\M_{\PP^3}(4m+1)$. The exact sequence
\begin{multline*}
0 = \Hom(\CC_P, \O_C) \lra \Ext^1_{\O_{\PP^3}} (\CC_P, \O_L(-1)) \isom \CC \lra \Ext^1_{\O_{\PP^3}} (\CC_P, \O_D) \\
\lra \Ext^1_{\O_{\PP^3}} (\CC_P, \O_C) \isom \CC
\end{multline*}
shows that $\Ext^1_{\O_{\PP^3}} (\CC_P, \O_D)$ has dimension $2$. Indeed, if this vector space had dimension $1$,
then its image in $\M_{\PP^3} (4m+1)$ would be a point. This, we saw above, is not the case.

Let $\DD_2 \subset \DD$ be the closed subset given by the condition $L \subset H$.
Equivalently, $\DD_2$ is given by the condition $C = L \cup C'$ and $P \in C'$ for a conic curve $C'$.
According to \cite[Proposition 4.10]{choi_chung_maican}, the set $\DD_2$ is irreducible of dimension $14$.
Indeed, let
\begin{equation}
\label{set_S}
\mathbf{S} \subset \Hilb_{\PP^2}(m+1) \times \M_{\PP^2}(3m+1)
\end{equation}
be the locally closed subset of pairs $(L, [\O_C(P)])$ for which $C = L \cup C'$ and $P \in C'$, for a conic curve $C' \subset \PP^2$.
According to \cite[Lemma 4.9]{choi_chung_maican}, $\mathbf{S}$ is irreducible.
The canonical map $\DD_2 \to \mathbf{S}$ is surjective and its fibers are irreducible of dimension $3$.


\section{The irreducible components}
\label{components}

Let
\[
\WW_0 \subset \Hom(3\O(-3), 5\O(-2)) \times \Hom(5\O(-2), \O(-1) \oplus \O)
\]
be the subset of pairs of morphisms equivalent to pairs $(\psi, \f)$ occurring in resolutions (\ref{sheaves_in_R}) and (\ref{sheaves_in_E}).
We claim that $\WW_0$ is locally closed.
To see this, consider first the locally closed subset $\mathbb{W}$ given by the following conditions:
$\psi$ is injective, $\f$ is generically surjective, $\f \circ \psi = 0$.
We have the universal sequence
\[
3\O_{\mathbb{W} \times \PP^3} (-3) \stackrel{\Psi}{\lra} 5\O_{\mathbb{W} \times \PP^3} (-2) \stackrel{\Phi}{\lra}
\O_{\mathbb{W} \times \PP^3}(-1) \oplus \O_{\mathbb{W} \times \PP^3}.
\]
Denote $\widetilde{\F} = \Coker(\Phi)$.
Corresponding to the polynomial $P(m) = 4m+1$ we have the locally closed subset
\[
\mathbb{W}_P = \{ x \in \mathbb{W},\ P_{\widetilde{\F}_x} = P \} \subset \mathbb{W}
\]
constructed when we flatten $\widetilde{\F}$, see \cite[Theorem 2.1.5]{huybrechts_lehn}.
Now $\WW_0 \subset \mathbb{W}_P$ is the subset given by the condition that $\widetilde{\F}_x$ be semi-stable,
which is an open condition, because $\widetilde{\F}_{| \mathbb{W}_P \times \PP^3}$ is flat over $\mathbb{W}_P$.
We endow $\WW_0$ with the induced reduced structure.
Consider the map
\[
\rho_0 \colon \WW_0 \lra \MM_0, \qquad (\psi, \f) \longmapsto [\Coker(\f)].
\]
On $\WW_0$ we have the canonical action of the linear algebraic group
\[
\GG_0 = \big( \Aut(3\O(-3)) \times \Aut(5\O(-2)) \times \Aut(\O(-1) \oplus \O) \big) / \CC^*
\]
where $\CC^*$ is identified with the subgroup
$
\{ (t \cdot \operatorname{id}, t \cdot \operatorname{id}, t \cdot \operatorname{id}), \ t \in \CC^* \}
$.
It is easy to check that the fibers of $\rho_0$ are precisely the $\GG_0$-orbits.
Let
\[
\WW_1 \subset \Hom (3\O(-3), 5\O(-2) \oplus \O(-1)) \times \Hom (5\O(-2) \oplus \O(-1), 2\O(-1) \oplus \O)
\]
be the locally closed subset of pairs of morphisms equivalent to pairs $(\psi, \f)$ occurring in resolution (\ref{sheaves_in_M_1}) and let
\[
\WW_2 \subset \Hom(\O(-4) \oplus \O(-2), \O(-3) \oplus 3\O(-1)) \times \Hom(\O(-3) \oplus 3\O(-1), 2\O)
\]
be the set of pairs given at \cite[Theorem 6.1(iii)]{choi_chung_maican}.
The groups $\GG_1$, $\GG_2$ are defined by analogy with the definition of $\GG_0$.
As before, for $i = 1, 2$, the fibers of the canonical quotient map $\rho_i \colon \WW_i \to \MM_i$ are precisely the $\GG_i$-orbits.

\begin{proposition}
\label{quotients}
For $i = 0, 1$, $\MM_i$ is the categorical quotient of $\WW_i$ modulo $\GG_i$.
The subvariety $\MM_2$ is the geometric quotient of $\WW_2$ modulo $\GG_2$.
\end{proposition}

\begin{proof}
The argument at \cite[Theorem 3.1.6]{drezet_maican} shows that $\rho_0$, $\rho_1$, $\rho_2$ are categorical
quotient maps. Since $\MM_2$ is normal (being smooth), we can apply \cite[Theorem 4.2]{popov_vinberg}
to conclude that $\rho_2$ is a geometric quotient map.
\end{proof}

Consider the closed subset $\WW_\el = \rho_0^{-1} (\EE) \subset \WW_0$.
Consider the restriction to the second direct summand of the map
\[
\O_{\WW_{\scriptscriptstyle \operatorname{ell}} \times \PP^3}(-1) \oplus
\O_{\WW_{\scriptscriptstyle \operatorname{ell}} \times \PP^3} \lra
\widetilde{\F}_{| \WW_{\scriptscriptstyle \operatorname{ell}} \times \PP^3}
\]
and denote its image by $\widetilde{\F}'$.
The quotient $[\O_{\WW_{\scriptscriptstyle \operatorname{ell}} \times \PP^3} \twoheadrightarrow \widetilde{\F}']$ induces a morphism
\[
\sigma \colon \WW_\el \lra \Hilb_{\PP^3}(4m).
\]
According to \cite[Examples 2.8 and 4.8]{chen_nollet}, $\Hilb_{\PP^3}(4m)$ has two irreducible components, denoted $H_1$, $H_2$.
The generic member of $H_1$ is a smooth elliptic quartic curve.
The generic member of $H_2$ is the disjoint union of a planar quartic curve and two isolated points.
Note that $H_2$ lies in the closed subset
\[
H = \{ [ \O \twoheadrightarrow \mathcal{S}] \mid \ \h^0(\mathcal{S}) \ge 3 \} \subset \Hilb_{\PP^3}(4m).
\]
Since $\sigma$ factors through the complement of $H$, we deduce that $\sigma$ factors through $H_1$.
By an abuse of notation we denote the corestriction by $\sigma \colon \WW_\el \to H_1$.

\begin{proposition}
\label{E_closure}
The sets $\DD_0$, $\DD_1$, $\DD_2$, $\DD$ and $\EE$ are contained in the closure of $\EE_0$.
The set $\DD$ is irreducible and $\DD_0$ is dense in $\DD$.
Moreover,
\[
\overline{\EE} \setminus \mathbf{P} = \EE \cup \DD = \EE \cup \DD', \qquad
\overline{\RR} \setminus (\overline{\EE} \cup \mathbf{P}) = \RR.
\]
\end{proposition}

\begin{proof}
Let $\EE_{\reg} \subset \EE_0$ be the open subset of sheaves with smooth support.
Let $H_{10} \subset H_1$ be the open subset consisting of smooth elliptic quartic curves.
For any $x \in H_1 \setminus H_{10}$ there is an irreducible quasi-projective curve $\Gamma \subset H_1$
such that $x \in \Gamma$ and $\Gamma \setminus \{ x \} \subset H_{10}$.
To produce $\Gamma$ proceed as follows. Embed $H_1$ into a projective space.
Intersect with a suitable linear subspace passing through $x$ to obtain a subscheme of dimension $1$
all of whose irreducible components meet $H_{10}$.
Retain one of these irreducible components and remove the points, other than $x$, that lie outside $H_{10}$

Notice that if $y = [\O \twoheadrightarrow \O_E]$ is a point in $H_{10}$, then $\sigma^{-1}\{ y \}$ is irreducible of dimension $1 + \dim \GG_0$.
Indeed,
\[
\sigma^{-1} \{ y \} = \rho_0^{-1} \{ [\O_E(P)], \ P \in E \}.
\]
Assume now that $x = [\O \twoheadrightarrow \O_E]$ where $E$ is the schematic support of a sheaf in $\EE \setminus \DD$.
We denote its irreducible components by $Z_0, \ldots, Z_m$.
Denote by $(\EE \setminus \DD)^0$ the open subset of sheaves of the form $\O_{E'}(P')$ with $P'$ lying outside $Z_1 \cup \ldots \cup Z_m$
and let $\WW^0$ be its preimage under $\rho_0$.
Denote by $\sigma_0$ the restriction of $\sigma$ to $\WW^0$.
Clearly, $\sigma_0^{-1}\{ y \}$ is irreducible of dimension $1 + \dim \GG_0$ and the same is true for $\sigma_0^{-1}\{ x \}$.
Thus, the fibers of the map $\sigma_0^{-1}(\Gamma) \to \Gamma$ are all irreducible of the same dimension.
By \cite[Theorem 8, page 77]{shafarevich} we deduce that $\sigma_0^{-1}(\Gamma)$ is irreducible.
Thus, $\rho_0(\sigma^{-1}(\Gamma))$ is irreducible, hence any sheaf of the form $\O_E(P)$, $P \in Z_0 \setminus (Z_1 \cup \ldots \cup Z_m)$,
is the limit of sheaves in $\EE_{\reg}$. The same argument applies to $\O_E(P)$ for $P$ belonging to exactly one of the components of $E$.
A fortiori, $\O_E(P)$ lies in the Zariski closure of $\EE_{\reg}$ for all $P \in E$.
We conclude that $\EE \setminus \DD \subset \overline{\EE}_0$.

Let $D$ be the union of a line $L$ and a planar irreducible cubic curve $C$, where $L$ and $C$ meet precisely at a regular point of $C$.
Take $x = [\O \twoheadrightarrow \O_D]$. Then
\[
\sigma^{-1} \{ x \} = \rho_0^{-1} \{ [\O_D(P)],\ P \in C \setminus L \}
\]
is irreducible of dimension $1 + \dim \GG_0$.
We deduce as above that any sheaf of the form $\O_D(P)$, $P \in C \setminus L$, is the limit of sheaves in $\EE_{\reg}$.
The set of sheaves of the form $\O_D(P)$ is dense in $\DD_0$.
We conclude that $\DD_0 \subset \overline{\EE}_0$.

Let $\DD^o \subset \DD \cap \EE = \DD \setminus \DD'$ be the open subset given by the condition that $P \notin L$.
Let $\sigma^o \colon \DD^o \to H_1$ denote the restriction of $\sigma$.
According to \cite[Theorem 5.2 (4)]{vainsencher}, there is an irreducible closed subset $\hat{\BB} \subset H_1$
whose generic member is the union of a planar cubic curve and an incident line.
Let $D$ be the schematic support of a sheaf in $\DD_2$.
According to \cite[Theorem 5.2 (5)]{vainsencher}, the point $x = [\O \twoheadrightarrow \O_D]$ belongs to $\hat{\BB}$.
By the same argument as above, there is an irreducible quasi-projective curve $\Gamma \subset \hat{\BB}$ containing $x$
such that the points $y \in \Gamma \setminus \{ x \}$ are of the form $[\O \twoheadrightarrow \O_{L \cup C}]$,
where $C$ is a planar irreducible cubic curve and $L$ is an incident line.
Notice that
\[
(\sigma^o)^{-1} \{ y \} = \rho_0^{-1} \{ [\O_{L \cup C}(P)], \ P \in C \setminus L \}
\]
is irreducible of dimension $1 + \dim \GG_0$.
Assume, in addition, that $D$ is the union of an irreducible plane conic curve $C'$ and a double line supported on $L'$.
Then
\[
(\sigma^o)^{-1} \{ x \} = \rho_0^{-1} \{ [\O_D(P)], \ P \in C' \setminus L' \}
\]
is irreducible of dimension $1 + \dim \GG_0$.
We deduce, as above, that $(\sigma^o)^{-1} (\Gamma)$ is irreducible, hence $\rho_0((\sigma^o)^{-1} (\Gamma))$ is irreducible,
and hence any sheaf of the form $\O_D(P)$, $P \in C' \setminus L'$, is the limit of sheaves in $\DD_0$.
But $\DD_2$ is irreducible,
hence the set of sheaves $\O_D(P)$ as above is dense in $\DD_2$.
We deduce that $\DD_2 \subset \overline{\DD}_0$. Thus $\DD_2 \subset \overline{\EE}_0$.

Recall from Proposition \ref{D_1_in_D_0} that $\DD_1 \subset \overline{\DD}_0$.
Since $\DD = \DD_0 \cup \DD_1 \cup \DD_2$,
we see that $\DD \subset \overline{\DD}_0 \subset \overline{\EE}_0$.

The inclusion $\overline{\EE} \setminus \mathbf{P} \subset \EE \cup \DD'$
follows from Theorem \ref{homological_conditions} and Proposition \ref{non-planar_M_1}.
Indeed, $\EE$ is closed in $\MM_0$.
The reverse inclusion was proved above.
Finally,
\[
\overline{\RR} \setminus (\overline{\EE} \cup \mathbf{P}) = \overline{\RR} \setminus (\EE \cup \DD' \cup \mathbf{P})
\subset \MM \setminus (\EE \cup \DD' \cup \mathbf{P}) = \MM_0 \setminus \EE = \RR.
\]
The reverse inclusion is obvious because by definition $\RR$ is disjoint from $\EE$, $\DD'$, $\mathbf{P}$.
\end{proof}

\noindent
From Proposition \ref{E_closure} we obtain the decomposition of $\M_{\PP^3}(4m+1)$ into irreducible components.

\begin{theorem}
\label{main_theorem}
The moduli space $\M_{\PP^3}(4m+1)$ consists of three irreducible components $\overline{\RR}$, $\overline{\EE}$ and $\mathbf{P}$.
\end{theorem}

\noindent
The intersections $\overline{\RR} \cap \mathbf{P}$, $\overline{\EE} \cap \mathbf{P}$, $\overline{\RR} \cap \overline{\EE}$
were described generically in \cite{choi_chung_maican}.
They are irreducible and have dimension $14$, $16$, respectively, $15$.
The generic member of $\overline{\RR} \cap \mathbf{P}$ has the form $[\O_C(P_1 + P_2 + P_3)]$, where $C$ is a planar quartic curve
and $P_1$, $P_2$, $P_3$ are three distinct nodes.
The generic point in $\overline{\EE} \cap \mathbf{P}$ has the form $[\O_C(P_1 + P_2 + P)]$,
where $C$ is a planar quartic curve, $P_1$ and $P_2$ are distinct nodes and $P$ is a third point on $C$.
The generic sheaves in $\overline{\RR} \cap \overline{\EE}$ have the form $\O_E(P)$, where $E$ is a singular $(2, 2)$-curve
on a smooth quadric surface and $P \in \sing(E)$.


\begin{thebibliography}{99}

\bibitem{ballico_huh} E. Ballico and S. Huh,
Stable sheaves on a smooth quadric surface with linear Hilbert bipolynomials.
The Scientific World Journal (2014), Article ID 346126.

\bibitem{chen_nollet} D. Chen and S. Nollet,
Detaching embedded points.
Algebra Number Theory {\bf 6} (2012), 731--756.

\bibitem{choi_chung_geometry} J. Choi and K. Chung,
The geometry of the moduli space of one-dimensional sheaves.
Sci. China Math. {\bf 58} (2015), 487--500.

\bibitem{choi_chung_moduli} J. Choi and K. Chung,
Moduli spaces of $\alpha$-stable pairs and wall-crossing on $\PP^2$.
J. Math. Soc. Japan {\bf 68} (2016), 685--709.

\bibitem{choi_chung_maican} J. Choi, K. Chung and M. Maican,
Moduli of sheaves supported on quartic space curves.
Michigan Math. J. {\bf 65} (2016), 637--671.

\bibitem{choi_maican} J. Choi and M. Maican,
Torus action on the moduli spaces of torsion plane sheaves of multiplicity four.
J. Geom. Phy. {\bf 83} (2014), 18--35.

\bibitem{drezet_maican} J.-M. Dr\'ezet and M. Maican,
On the geometry of the moduli spaces of semi-stable sheaves supported on plane quartics.
Geom. Dedicata {\bf 152} (2011), 17--49.

\bibitem{freiermuth_trautmann} H.-G. Freiermuth and G. Trautmann,
On the moduli scheme of stable sheaves supported on cubic space curves.
Amer. J. Math. {\bf 126} (2004), 363--393.

\bibitem{huybrechts_lehn} D. Huybrechts and  M. Lehn,
\emph{The Geometry of Moduli Spaces of Sheaves.}
Aspects of Mathematics E31, Vieweg, Braunschweig, 1997.

\bibitem{iena} O. Iena,
On the singular $1$-dimensional planar sheaves supported on quartics.
Rend. Ist. Mat. Univ. Trieste {\bf 48} (2016), 565--586.

\bibitem{lepotier} J. Le Potier,
Faisceaux semi-stables de dimension $1$ sur le plan projectif.
Rev. Roum. Math. Pures Appl. {\bf 38} (1993), 635--678.

\bibitem{maican_duality} M. Maican,
A duality result for moduli spaces of semistable sheaves supported on projective curves.
Rend. Semin. Mat. Univ. Padova {\bf 123} (2010), 55--68.

\bibitem{maican_international} M. Maican,
The homology groups of certain moduli spaces of plane sheaves.
Int. J. Math. {\bf 24} (2013), Article ID 1350098, 42 p.

\bibitem{maican_sciences} M. Maican,
On the homology of the moduli space of plane sheaves with Hilbert polynomial $5m+3$.
Bull. Sci. Math. {\bf 139} (2015), 1--32.

\bibitem{popov_vinberg} V. Popov and E. Vinberg,
\emph{Invariant Theory}.
A.~Parshin, I.~Shafarevich (Eds.), G.~Kandall (Trans.), Algebraic Geometry IV,
Encyclopedia of Mathematical Sciences v. 55, Springer Verlag, Berlin, 1994.

\bibitem{shafarevich} I. Shafarevich,
\emph{Basic Algebraic Geometry I}.
Second, Revised and Expanded Edition,
Springer-Verlag, Berlin, Heidelberg, 1994.

\bibitem{vainsencher} I. Vainsencher and D. Avritzer,
Compactifying the space of elliptic quartic curves.
Complex projective geometry, Sel. Pap. Conf. Proj. Var., Trieste/Italy 1989,
and Vector Bundles and Special Proj. Embeddings, Bergen/Norway 1989,
Lond. Math. Soc. Lect. Note Ser. 179, 47-58 (1992).


\end{thebibliography}
\end{document}